 \documentclass[a4paper,12pt]{amsart}
\input{amssym.def}
\input{amssym.tex}

\usepackage{amsmath,amscd}
\usepackage{verbatim}   

\newtheorem{theorem}{Theorem}[section]
\newtheorem*{theoremFC}{Theorem FC}
\newtheorem{lemma}[theorem]{Lemma}
\newtheorem{corollary}[theorem]{Corollary}
\newtheorem{proposition}[theorem]{Proposition}

\theoremstyle{definition}
\newtheorem{definition}[theorem]{Definition}
\newtheorem{definitions}[theorem]{Definitions}

\numberwithin{equation}{section}

\newcommand{\Ker}{\operatorname{Ker}}

\newcommand\Hobs{{H^{\text{\rm obs}}_{A,\theta}}}
\newcommand\Hctr{{H^{\text{\rm ctr}}_{A,\theta}}}
\newcommand\beqn{\begin{equation}}
\newcommand\neqn{\end{equation}}

\newcommand{\al}{\alpha}   
\newcommand{\be}{\beta}   
\newcommand{\de}{\delta}   
\newcommand{\De}{\Delta}   
\newcommand{\ga}{\gamma}   
\newcommand{\Ga}{\Gamma}   
\newcommand{\tht}{\theta} 
\newcommand{\om}{\omega}   
\newcommand{\si}{\sigma}   
\newcommand{\Hnty}{H^\infty} 
\newcommand{\wt}{\widetilde } 
\newcommand{\wh}{\widehat } 
\newcommand{\cO}{\mathcal{O} } 
\newcommand{\cW}{\mathcal{W} } 

\newcommand{\sgn}{\operatorname{sgn}}
\newcommand{\spec}{\operatorname{spec}}

\newcommand{\margp}[1]{}

\newcommand{\la}{\lambda}
\newcommand{\RR}{\Bbb R}

\newcommand{\CC}{\Bbb C}

\newcommand{\ZZ}{\Bbb Z}

\newcommand{\DD}{\Bbb D}

\newcommand{\Q}{\mathcal{Q}}

\newcommand{\Lom}{\mathcal{L}}

\begin{document}

\title[$H^\infty$-functional calculus and models of Nagy-Foia\c{s} type]
{$H^\infty$-functional calculus and models of Nagy-Foia\c{s} type for sectorial operators}

\author{Jos\'{e} E. Gal\'{e}}
\address{Departamento de Matem\'{a}ticas and I. U. M. A.,
Universidad de Zaragoza, Zaragoza 50009, Spain} \email{gale@unizar.es}

\author{Pedro J. Miana}
\address{Departamento de Matem\'{a}ticas and I. U. M. A.,
Universidad de Zaragoza, Zaragoza 50009, Spain} \email{pjmiana@unizar.es}
\thanks{The first and  second author have been supported by Project E-64, Gobierno de  Arag\'on and Projects
MTM2004-03036 and MTM2007-61446, DGI-FEDER, of the MCYT, Spain.}

\author{Dmitry V. Yakubovich}
\address{Departamento de Matem\'{a}ticas,
Universidad Autonoma de Madrid, Cantoblanco 28049 (Madrid) Spain
\enspace and
\newline
\phantom{r} Instituto de Ciencias
Matem\'{a}ticas (CSIC - UAM - UC3M - UCM)}
\email{dmitry.yakubovich@uam.es}
\thanks{The third author has been supported by the Ram\'on and Cajal Programme (2002) and the
Project MTM2008-06621-C02-01, DGI-FEDER, of the Ministry of Science and Innovation, Spain.}


\begin{abstract}
We prove that a sectorial operator admits an $H^\infty$-functional
calculus if and only if it has a functional model of Nagy-Foia\c{s} type. Furthermore, we give a concrete formula for the
characteristic function (in a generalized sense) of such an operator.
More generally, this approach applies to any sectorial operator by passing to a  different norm (the McIntosh square function norm).
We also show that this quadratic norm
is close to the original one, in the sense that there is only a logarithmic gap between them.

\medskip

\noindent {\bf M.S.C.(2000):}  Primary: 47A45; Secondary: 47A60, 32A10.

\end{abstract}
\maketitle
\section{Introduction }
 Let $H$ be a separable Hilbert space
 and let $L(H)$
 \margp{$L(H)$ en vez de $\mathcal{L}(H)$ (transformada de Laplace !!)}
 be the Banach algebra of linear and bounded operators on $H$.
 For $0<\omega <\pi$, we put
 $S_\omega:=\{z\in \CC\,\,;\,\, \vert \hbox{arg}(z)\vert \le
 \omega\}\cup\{0\}$. Let $A$ be a closed operator with domain $D(A)$ and spectrum
 $\sigma(A)$. Put $\rho(A):=\CC\setminus\sigma(A)$. The operator $A$ is said to be {\it sectorial} of
 type $\omega$ if $\sigma(A)\subset S_\omega$ and, for each $\theta$ with
 $\omega<\theta<\pi$, there exists $C_\theta$ such that
 \beqn 
 \label{sect_cond}
 \Vert (z-A)^{-1}\Vert \le \frac{C_\theta}{\vert z\vert}, \qquad
 z\not \in S_\theta.
 \neqn 

Each operator of this type has a decomposition $A=0\oplus A_0$
with respect to some direct sum representation,
$H=\hbox{Ker}(A)\oplus H_0$, where  $\hbox{Ker}(A_0)=\{0\}$ and $A_0$
has a dense range. From now on, we assume that $A$ itself is
one-to-one, which
is equivalent to
the fact that it has dense range, see Remark 2.6 in \cite{leMeW}.

Take $\tht\in(\om,\pi)$ and denote by $S^\circ_\tht$ the interior of the sector $S_\tht$.
Following \cite{McInt1}, we define  the class $\Psi(S^\circ_\tht)$ as being formed by all holomorphic functions
$f\in {\hbox{Hol}}(S_\tht^\circ)$ such that
\beqn\label{Psi}\vert
f(z)\vert \le C{\frac{\vert z\vert^s} {1+\vert z\vert^{2s}}} , \quad \hbox{ for all } z\in S_\tht  \hbox{ and some } s, C>0.
\neqn

The operator $A$ admits
\margp{Cambiado}
a functional
calculus (the so-called Dunford-Riesz calculus), which is
constructed on the basis of the Cauchy operator-valued formula,
and which provides us with bounded operators $\psi(A)$ for
functions $\psi\in \Psi(S^\circ_\tht)$, $\tht>\omega$, see
\cite{McInt1}:
\begin{equation}
\label{DuRi}
\psi(A):=\frac 1{2\pi i}\int_\gamma(zI-A)^{-1}\psi(z)\ dz,
\end{equation}
where $\gamma$ is the contour defined
\margp{**}
as $\gamma(t)=-te^{i\tht'}$
if $-\infty<t\leq0$ and $\gamma(t)=te^{-i\tht'}$, if $0\leq t<\infty$,
for $\tht>\tht'>\omega$. The above calculus can be extended to
functions in $H^\infty(S^\circ_\tht)$ (that is, bounded and
analytic functions on $S^\circ_\tht$) by the formula
$f(A):=\varphi(A)^{-1}(f\varphi)(A)$, $f\in
H^\infty(S^\circ_\tht)$, where $\varphi(z)=z(1+z)^{-2}$, although
$f(A)$ may well be possibly unbounded \cite{McIYa}. The operator
$A$ is said to have a $H^\infty(S^\circ_\tht)$-functional calculus
if $f(A)\in L(H)$ for every $f\in
H^\infty(S^\circ_\tht)$. In this case, the mapping $f\mapsto f(A)$ is a
bounded Banach algebra homomorphism, see \cite[Proposition~ 5.3.4]{Haase}.

We refer to \cite[Theorem 3.3]{Che_Gale} for a discussion of the uniqueness of the
\margp{Referencia Che Gale}
continuation of the $H^\infty(\Omega)$-functional calculus from
the set of rational functions to different subalgebras of
$H^\infty(\Omega)$, where $\Omega$ is a disc or a simply connected domain,
and to \cite[Theorem 2.6]{Uiterd} for a uniqueness result in
the context of a functional calculus, which maps
analytic functions on $S^\circ_\tht$
that may have polynomial growth at $0$ and at $\infty$ into
the set of closed operators on $H$.

For every non-zero $\psi \in \Psi(S^\circ_\tht)$, set $\psi_t(z):=
\psi(tz)$ for $t>0$ and $z\in S_\tht^0$. Define the norm
\begin{equation}\label{normMC}
\| x\|_A:=\left(\int_0^\infty\Vert \psi_t(A)x\Vert^2\; \frac{dt} t
\right)^{\frac 12}
\end{equation}
on the linear manifold $H_c$ of those $x\in H$ for which this
expression is finite, and let $H_A$ denote the Hilbert space
obtained as the completion of $H_c$ with respect to the above
norm. Different choices of $\tht$ and $ \psi$ give rise to
equivalent norms, so to the same space $H_A$ \cite{leMeW}.  It is
known that
\begin{equation}\label{FC_A}
\Vert f(A)x\Vert_A\le \Vert f\Vert_{\infty,\tht}\, \Vert x\Vert_A, \qquad x\in H_A, f\in H^\infty(S^\circ_\tht),
\end{equation}
where $\Vert f\Vert_{\infty,\tht}$ denotes the sup-norm of $f$ on
$S^\circ_\tht$, see \cite{McIYa}, or  \cite[Theorem 3.1]{leMeW}.
Then the operator $A$ has a bounded
$H^\infty(S^\circ_\tht)$-functional calculus if and only if  the
norm $\Vert \quad \Vert_A$ is equivalent to the norm in $H$ and
$H_A=H$. These questions have much relationship with the Kato
problem and the boundedness of the Cauchy integrals on Lipschitz
curves; see for example \cite{AuMcIntNahm, Haase} and
references therein. We refer also to the survey \cite{Wei_L}
for a brief introduction to the $H^\infty$ calculus for sectorial operators.

Another setting where the $H^\infty$-functional calculus appears
naturally is that of the Nagy-Foia\c{s} functional model of
Hilbert space operators \cite{SzNF}. This model is constructed
originally for contractive or dissipative operators in the unit
disc $\DD$ or in the upper half-plane, respectively. Let us
explain the above in some
more detail. Assume that $T$ is a
contraction on $H$
\margp{$\lim_{n\to\infty}T^n\,x=0$ for all $x\in H$}
such that $\lim_{n\to\infty}T^n\,x=0$
for all $x\in
H$. The existence of a Nagy-Foia\c{s} model for $T$ means that
$T$ can be realized as the multiplication operator $\widehat{M}_z$
on the quotient space $H^2(\DD,E)/\delta H^2(\DD,F)$, through a
Hilbertian isometry $V\colon H^2(\DD,E)/\delta H^2(\DD,F)\to H$.
Here $E$ and $F$ are auxiliary Hilbert spaces known as defect
spaces for $H$, and $\delta$ in $H^\infty(\DD; L(F,E))$
is the (unique) so-called Nagy-Foia\c{s} characteristic function
for $T$, see \cite{SzNF}. Then the $H^\infty$-functional calculus
for $T$ follows immediately by putting
$f(T):=V\circ\widehat{M}_{f(z)}\circ V^{-1}$ for all $f\in
H^\infty(\DD)$.

\medskip
A theory about models of Nagy-Foia\c{s} type in rather general
domains  $\Omega\subset\CC$ has been recently established in
\cite{Yak}.
%
Let us explain its connections with the linear control theory,
using the notation of \cite{YaPa}.
Recall that a linear observation  system is given by
$$
\begin{cases}
x'(t)= - Ax(t) \in X,& 0\leq t<\infty; \cr
x(0)=a\in X; \cr y(t)= Cx(t)\in Y, &0\leq t<\infty.
\end{cases}
$$
Here $Y$ is the output Hilbert space, $X$ is a Hilbert space which
has the meaning of the system state space, and $A$ is assumed to
be the generator of a $C_0$ semigroup on  $X$. The function
$y\colon[0,\infty)\to Y$ is called the output of the system. Then
one can define the mapping $\cO_{A,C}: X\to
\mbox{Hol}\,(\rho(A);Y)$ by
$$
\cO_{A,C} x(z):=C(z-A)^{-1} x, \qquad z\in\rho(A).
$$
Notice
\margp{**}
that $-\cO_{A,C} x(-z)=\Lom(y)(z)$ for $\Re z>0$,
where $\Lom(y)$ is the Laplace transform of $y$.

The mapping $\cO_{A,C}$ is called the observation map. It will be shown later
that this map gives rise to the \textit{observation model} of $A$.

A linear control system on $\Omega$ is defined by
\margp{eqn**}
$$
x'(t)=-Ax(t)+Bu(t),\ \  -\infty<t\leq 0
$$
for $u\in L^2((-\infty,0),U)$ of compact support, where $U$ is the
input Hilbert space of the system and $B\colon D(B)\to X$ is a
closed operator. Assume that there exists $\mathcal{C}_{A,B}\colon
L^2((-\infty,0),U)\to X$,
which is a
continuous extension of  $u\mapsto
x(0)$. Then the controllability map $W_{A,B}$ is defined as
\margp{**}
$W_{A,B}:=\mathcal{C}_{A,B}\circ\Lom^{-1}\circ \operatorname{inv}$, where
$\operatorname{inv} f(z)=f(-z)$, see \cite{YaPa}, Section 5.
This map gives rise to the \textit{control model} of $A$.
Formulae of $W_{A,B}$
type will be given in detail in Section 2 below, for sectorial $A$.
These two models
are in a certain duality, as is explained in
\cite{Yak}, \cite{YaPa}.
In fact, one can consider the observation and the control model
both for $A$ and for $A^*$, which gives four
different models.


\medskip
The triple $(U, Y, X)$ of Hilbert spaces and
\margp{**}
corresponding mappings $\cO_{A,C}$, $W_{A,B}$ previously
considered can be defined independently of differential equations,
and they can be taken as starting point to construct functional
models for the operator $A$,  in a more abstract context. More
precisely, under suitable conditions on $\Omega$ and $A$, such
models are built on the corresponding representation Hilbert spaces
$\Hobs$ and $\Hctr$. There is a natural isomorphism
between these spaces.
It is shown that
$A$ is similarly equivalent (but not
necessarily unitarily equivalent) to its model, which is,
basically,  the
multiplication operator either on $\Hobs$ or on $\Hctr$.

When implementing
\margp{cambios}
this procedure, it is very important to make a good
choice of operators $B$ and $C$ for a given operator $A$.
The generalized
characteristic function $\delta$ need
not be unique, and must be well chosen; see \cite{Yak}, where several
examples  are given. In \cite{YaPa}, this
theory has been carried out to construct models on certain
parabolic domains, which apply to nondissipative perturbations of
unbounded self-adjoint operators on a Hilbert space.

The aim of the present  paper is to give a concrete construction
of control and observation functional models of Nagy-Foia\c{s}
type for {\it every} sectorial operator $A$ on a
\margp{**}
Hilbert space
 of type $\om$ in sectors $S^\circ_\tht$, $\om<\tht<\pi$.
(To be more
\margp{**}
precise, the model is constructed for the natural extension
of $A$ onto $H_A$.)

 The
relationships between sectorial operators and control theory are
not new, see for instance \cite{Arli}, \cite{Lemerdy}. In
\cite[Theorem 4.1]{Lemerdy}, see also \cite[p. 204]{leMeW}, Le Merdy considers
the admissibility of operators $C$ as above in the case that $-A$
is the infinitesimal generator of a bounded analytic
$C_0$-semigroup. He shows that $C=\sqrt A$ is an admissible
operator for $A$ if and only if $A$ has a square function
estimate, that is, $\Vert x\Vert_A\leq K\Vert x\Vert$, ($x\in H$),
for some constant $K>0$.

Inspired by the above result, we shall prove here that the particular choice
\margp{eqn**}
$$
B:= 2\sqrt{A}, \qquad C:=\sqrt{A}, \qquad U=Y:=H,
\qquad X:=H_A,
$$
gives rise to a model of $A$ of  Nagy-Foia\c{s} type in $S_\theta^\circ$.
Moreover, for $\tht\in (\om,\pi)$ and $\alpha>0$ such
that $\alpha \tht<\frac\pi 2$, let $\delta_\alpha$ be
the operator function defined by
\beqn
\label{delta-al}
\delta_\al(z):=\frac 1\al\;{A^{\alpha}-z^\alpha\over A^\alpha+z^\alpha},
\qquad z\in S_\tht.
\neqn

In the definition of $z^\al$ and $A^\al$
we mean that $z^\alpha$ is the continuous branch of the function
$z^\al$ on $\CC\setminus (-\infty,0]$ such that $(1)^\al=1$.
As it will be explained below,
$\de_\al$ is in $H^\infty\big(S^\circ_\tht; L(H)\big)$.

We shall show that \textit{for any $\tht$ and $\al$ as above, the
operator $A$, considered as an operator on $H_A$,  possesses a
Nagy--Foia\c{s} type model in the open sector $S^\circ_\tht$, and
$\de_\al$ is one of its generalized Nagy--Foia\c{s} characteristic
functions in the sense of} \cite{Yak} (see Theorems
\ref{Ctr-mod} and \ref{Obs-mod} below). One of the stages in the proof
of this result is
\margp{cambio}
to prove equalities
$$
H_A=\Hctr=\Hobs, \qquad \forall\tht\in (\om,\pi)
$$
(with the equivalence of the corresponding norms).
Thus we obtain that quadratic estimates
related to sectorial operators, which on the other hand have
revealed significant in the treatment of several analytic
questions \cite{AuMcIntNahm}, are also in the basis of control
theory and functional models of Nagy--Foia\c{s} type.

In particular it follows from the above that for every one-to-one
operator $A$ of type $\om$, and any $\tht\in(\om,\pi)$, the
following three properties are equivalent, see Theorem FC and
Theorem \ref{eqiv conds} below :

\

\noindent (1) There exists a $\Hnty$ functional calculus
in $S_\tht^\circ$ for $A$.

\noindent (2) The equality
$H_A=H$ holds.

\noindent (3) There exists a Nagy--Foia\c{s} type functional model in
$S_\tht^\circ$ for $A$, and $\de_\al$ is
\margp{**}
a generalized characteristic function of $A$.

\

\noindent As a consequence, the existence of a Nagy--Foia\c{s}
functional model for $A$ in sectors depends on the equality
$H_A=H$ (the equivalence between (1) and (2) above was well known,
as it has been already mentioned).
How distant the spaces (and their norms) $H_A$ and $H$ are is a
question of independent interest, which we also consider here.
\margp{cambiado}
In Theorem \ref{thm log}, we
show that
the gap between these two norms is, in a certain sense, of logarithmic order
and hence is small.

\

The plan of the paper is as follows. In Section 2 we introduce the
main concepts (spaces and operators)  associated with a sectorial
operator, from the point of view of (abstract) control theory, and
state the main theorems of the article. One of the  main tools in
the paper will be an operator ${\mathcal J}_{\al}$ of Hankel type
defined on the Hardy-Smirnov class which, in principle, depends on
the parameter $\alpha\in(0, \pi/\omega)$. A key point in our
arguments \margp{unas palabras quitadas} is that such an operator
{\it does not} depend on $\alpha$ in the above range, indeed. To
prove this, we need to establish some results about analytic
extensions
of operator-valued functions $\delta_{\al}$
 to sectors $S^\circ_\tht$. Such extensions are
studied in Section \ref{complex}. In Section \ref{like}, we
introduce a Hankel-like operator ${\mathcal J}_{\al}$ and prove
its basic properties, including a kind of independence with
respect to $\al$.   Control and observation (Hilbert) spaces and
operators associated with a sectorial operator are studied in
Section \ref{s aux}. We proceed with a collection of preparatory
results which eventually allows us to prove the concrete
isomorphism results for
 the control and observation spaces. Section \ref{Proofs main res} is devoted to
culminate the proofs of the main theorems, stated in Section
\ref{jugoso}. Finally, Section \ref{Sec_final} contains some remarks and
comments.

We also would like to mention that \cite{Yak} contains a construction
of a model of Nagy-Foia\c{s} type for an arbitrary generator of
a $C_0$ group on a Hilbert space in a suitable vertical strip.
It follows, in particular, that these operators always admit an $H^\infty$ calculus in
a strip, which has been already known, due to a result by
Boyadzhiev and deLaubenfels \cite{Boy_deLaub} (the third author, unfortunately,
was not aware of this work at that moment).

\section{Main results}\label{jugoso}

Let us introduce the following notation.
Let $H$ denote a separable Hilbert space. (All Hilbert spaces appearing in this article are supposed to be separable.)
For a densely defined closed operator $T$ on $H$ with
trivial kernel we define the Hilbert space
$TH$ as the set of
formal expressions $T x$, where $x$ ranges
\textit{over the
whole space} $H$.
Put $\|Tx\|_{TH}:= \|x\|_H$ for all
$x\in H$.
If $T^{-1}\in  L(H)$, then the
formula $x=T(T^{-1}x)$ allows one to interpret $H$ as a linear submanifold
of $TH$. If $T$ is bounded, then, conversely, $TH$ is a
a linear submanifold of $H$.

Let $A$ be a sectorial operator on a Hilbert space $H$. We assume
\margp{are to -> can be}
concepts and notation of the above section. Our first result
\margp{quitada una frase por ser redundante}
concerns the comparison between the norms of the Hilbert spaces
$H$ and $H_A$.  Related
results
can be found
in \cite{deLau}, Corollary
4.7 and Theorem 7.1 (b).

In the following result $Log(z)$ is the principal branch of the
logarithm with argument in $[-\pi,\pi)$.
\begin{theorem}
\label{thm log} Put $\Lambda_k(z):= \hbox{Log(z)}+2k\pi i$ for
$z\in \CC\backslash(-\infty, 0]$. Then
for any $r>{1\over 2}$ and any $k\in\ZZ$, $k\ne0$ we have
$$
\Lambda_k(A)^{-r}H\subseteq H_A\subseteq
\Lambda_k(A)^{r}H
$$
and there is a
\margp{$C_t\to C_{k,r}$}
constant $C_{k,r}>0$ such that
  \begin{equation}\label{Lambda}
   C_{k,r}^{-1}\Vert  \Lambda_k(A)^{-r}x\Vert\le \Vert x\Vert_A\le C_{k,r}\Vert
  \Lambda_k(A)^{r}x\Vert,
  \end{equation}
  for all $x\in  \Lambda_k(A)^{-r}H$. The space
  $\Lambda_k(A)^{-r}H$ is dense in $H_A$.
  \end{theorem}

\bigskip

In the above result, the spaces $\Lambda_k(A)^{\pm r}H$
must be understood in the sense prior to the theorem, for
$T=\Lambda_k(A)^{r}$ (or $T=\Lambda_k(A)^{-r}$). Also, the
statement could be alternatively written saying that the linear
operators $\Lambda_k(A)^{-r}: H\to H_A$ and  $\Lambda_k(A)^{-r}:
H_A\to H$ are bounded.

Notice also that the operator $i\sgn(k)\Lambda_k(A)$
is $\nu$-sectorial for any $\nu>\pi/2$. This follows from the Nollau's
estimate for the resolvent of $\Lambda_k(A)$; see for example
\cite[Proposition 3.5.3]{Haase}. The spectrum of
$i\sgn(k)\Lambda_k(A)$
does not contain the origin.
It follows that for any $k\ne0$,
$\Lambda_k(A)^{-r}$
is bounded on $H$.

Now let us introduce
\margp{frases reescritas}
the main notions needed for the construction of the
Nagy--Foia\c{s} model, adapted to the context
of a sectorial operator.
We start by recalling the part of this construction in
the form that was given in \cite{Yak, YaPa},  which will be the
most convenient for us.

Let $\Omega$ be a complex domain with a rectifiable boundary
$\partial \Omega$. For a Hilbert space $H$, let $E^2(\Omega; H)$
denotes  the \textit{vector-valued Hardy--Smirnov class}, that is,
the set of functions $u\in \mbox{Hol}\,(\Omega;H)$ such that there
exists an increasing sequence $\{\Omega_n\}$ of bounded domains
with smooth boundary satisfying
$\overline \Omega_n\subset
\Omega$, $\bigcup_n\Omega_n=\Omega$, and
$$
\sup_n\int_{\partial\Omega_n}\|u(z)\|_H^2\,|dz|<\infty.
$$
Functions in $E^2(\Omega; H)$ have strong non-tangential
boundary limit values a.e.\hskip1.5mm on $\partial\Omega$, which allow one to
identify this space with a closed subspace of
$L^2(\partial\Omega; H)$. This gives $E^2(\Omega; H)$ a Hilbert space structure.
 We refer to \cite{Duren} and
\cite{SzNF} for a background.

Let $\omega$ be in $(0,\pi)$ the type of the sectorial operator
$A$. Fix an angle $\tht\in(\om,\pi)$. Throughout the paper, the
class $E^2(\Omega; H)$ will be considered for $\Omega=S^\circ_\tht$
and, in integrals over the boundary $\partial S_\theta$ of
$S_\theta$, we give $\partial S_\theta$ the parametrization
\margp{counter- clockwise}
counterclockwise around $S_\tht\cup\{\infty\}$. As indicated above, we
take in this case the operators $B:= 2\sqrt{A}$, $C:=\sqrt{A}$ and
Hilbert spaces $U=Y:=H$, $X:=H_A$, in
\margp{\quad $X:=H_A$}
the general framework of
\cite{Yak}, \cite{YaPa}. Then we get the \textit{observation}
operator $\cO: H\to \mbox{Hol}\,(\rho(A);H)$ associated with $A$,
given by
\begin{equation}
\label{def_cO}
(\cO x)(z)=\sqrt{A} (z-A)^{-1}x, \qquad z\in\rho(A), x\in H,
\end{equation}
and  the \textit{control} operator $W_\tht\colon
E^2(S_\tht^\circ,H)\to H$ for $A$ defined as
$$
W_\tht(u)
:=
{1\over \pi i}\int_{\partial
S_\theta}(\xi-A)^{-1}\sqrt{A}u(\xi)\,d\xi,
$$
for every $u\in E^2(S_\theta^\circ; H)$.

We shall need to show that $\cO$ and $W_\tht$ are bounded
as operators between suitable spaces. For this, set
$$
T_A:=\frac{\sqrt{A}}{1+A}.
$$
Since  $T_A=f(A)$ where
$f(z):=\sqrt{z}(1+z)^{-1}$ is in $\Psi(S^\circ_\tht)$, $T_A$ is bounded.
It is also clear that $\ker T_A=0$ and that  Im$T_A^n$ is dense in $H$ for every
$n>0$, \cite[p. 143]{Haase}.

\begin{definition} 
We define the spaces $H_n:=T_A^n H$, $n\in \mathbb{Z}$, by giving them
the meaning explained at the beginning of this section.
\end{definition}
If $n>0$, then $H_n$ is a dense linear subset in $H$.
If $n<0$, then $H$ is a dense linear subset of $H_n$.
We have a chain of dense imbeddings
$$
\dots\subset H_2\subset H_{1}\subset H_{0}=H\subset
H_{-1}\subset H_{-2}\subset \dots\quad.
$$
Then the mapping ${\mathcal O}$ extends to $H_{-1}$ by putting
$$
(\cO x)(z):=\frac{\sqrt{A}}{1+A}\,\frac{1+A}{z-A}\;x
\qquad (z\in\rho(A), x\in H_{-1}),
$$
and $\cO:H_{-1}\to \hbox{Hol}(\rho(A); H)$ is clearly continuous.
This map is injective.
Indeed, if $x\in H_{-1}$ satisfies
\margp{\quad ${\mathcal O}_\theta$ --> $\mathcal O$}
$\mathcal O x(\lambda)=0$ for some
$\lambda\in \CC\backslash S_\theta$ then
$$
x=(\lambda-A)(\sqrt{A})^{-1}\cO x(\lambda)=0.
$$
It can be also proven that the linear mappings
$W_\tht: E^2(S_\theta^\circ; H)\to H_{-1}$,
$\mathcal O\colon H_{1}\to E^2(\CC\backslash S_\theta; H)$
\margp{\quad lo mismo}
are well-defined and continuous (see Proposition \ref{prop1}
below). Then we can introduce the appropriate control and
observation spaces for $A$.

\begin{definitions}
\label{ctrobs} \textit{The control space} for $A$ is the Hilbert space
$\Hctr$
obtained as the range of $W_\theta$ in $H_{-1}$ with the range norm,
$$
\Vert x\Vert_{A,\tht,\text{ctr}}:=\min\{ \Vert u \Vert_{E^2(S_\theta^\circ;
H)}\, ; \,\,
x=W_\theta(u)\}.
$$
Put $\cO_\tht x:=\cO x\vert_{\CC\setminus S_\tht}$. \textit{The observation space}
for $A$ is
$$
\Hobs
:=\big\{x\in H_{-1}\,\, ;\,\, {\mathcal O}_\theta x\in E^2(\CC\backslash{S_\theta};
H)\big\},
$$
with the norm $\Vert x\Vert_{A,\tht,\text{obs}} := \Vert {\mathcal O}_\theta
x\Vert_{E^2(\CC\backslash{S_\theta}; H)}$.
\margp{he quitado este parrafo}
\end{definitions}

We have the following result.


\begin{theorem} \label{main}
The spaces $\Hctr$ and $\Hobs$
do not depend on $\theta$ and coincide with $H_A$ for any
$\theta \in (\omega, \pi)$. The norms in all these spaces are
mutually equivalent.
\end{theorem}

For $\al>0$ such that $\al\tht<\pi/2$ let $\de_\al$ be the operator-valued function
given by (\ref{delta-al}). The fractional power $A^\al$ is $\om\al$-sectorial
\cite{Haase}, so in particular
$\si(A^\al)\subset S_{\om\al}$. It follows that
$$
\de_\al(z):=\frac 1 \al \,(1- 2z^\al(A^\al+z^\al)^{-1}), \qquad z\in S_\tht,
$$
is an $ L(H)$-valued function on $S^\circ_\tht$
of class $H^\infty$.
It is also easy to see that
$\de_\al^{-1}(z)$ exists and is uniformly bounded for
$z\in\partial S_\tht$ with $z\ne0$.
Hence $\de_\al E^2(S_\theta^\circ; H)$ is a closed subspace of
$E^2(S_\theta^\circ; H)$.

For any $x\in H_A$, we have
$y:=(\sqrt A)^5(1+A)^{-3}x=\sqrt A(1+A)^{-1}(1-(1+A)^{-1})^2x\in H$, which means that $Ax=T_A^{-3}y$ is defined at least as a member of $H_{-3}$. We
define a (possibly unbounded) operator $\wt A$ on
$H_A$ by
$$
D(\wt A):=\{x\in H_A\,\, ; \,\, Ax\in H_A\},
$$
and $\wt Ax=Ax, $  $x\in D(\wt A)$. Then $\wt A$ is an
$\omega$-sectorial operator with trivial kernel and
$\sigma(A)=\sigma(\tilde A)$,
see Proposition \ref{Prop:spectr} below.  The following
theorem provides a concrete model of Nagy--Foia\c{s} type
of $\wt A$, which can be called a control model according to
the terminology of \cite{YaPa}.

\begin{theorem} 
\label{Ctr-mod}
Suppose, as above, that
$\tht\in(\om,\pi)$, $\al>0$, $\al\tht<\frac \pi2$.
Then
we have the following.
\begin{itemize}
\item[(i)] $\operatorname{Ker} W_\theta=\delta_\alpha E^2(S_\theta^\circ; H)$.

\item[(ii)]
Consider
\margp{**}
the quotient space
$$
{\mathcal Q}(S_\theta^\circ, \delta_\alpha):=E^2(S_\theta^\circ; H)/
\delta_\alpha E^2(S_\theta^\circ; H)
$$
and the corresponding
quotient operator
$\wh W_\theta: {\mathcal Q}(S_\theta^\circ, \delta_\alpha)\to H_A$.
Then
$\wh W_\tht$
is an isomorphism of
${\mathcal Q}(S_\theta^\circ, \delta_\alpha)$ onto
$H_A$.

\item[(iii)]
Let
\margp{**}
$\wh M_z$ be the quotient multiplication operator  on
\margp{he cambiado el orden: $\sigma=[s], \rho=[r], s(z)=z r(z)$}
${\mathcal Q}(S_\theta^{\circ}, \de_\al)$ given by
$$
\begin{aligned}
\wh M_z
\rho\:=\sigma,
\quad & \text{if }
\;\; \exists\, r, s\in E^2(S^\circ_\tht; H) \\
& \text{such that } \sigma=[s], \rho=[r], \; s(z)=z r(z)
\end{aligned}
$$
(here $[r]\in {\mathcal Q}(S_\theta^\circ, \de_\al)$
denotes the coset that corresponds to $r$).
We put $D(\wh M_z)$ to be the set of
all cosets $[r]$ such that there exist $r$, $s$ as above.
Then $\wh M_z$ is
well-defined. It is a closed operator
on ${\mathcal Q}(S_\theta, \de_\al)$
with dense domain, and
$\widehat{W_\theta}$ intertwines
$\wh M_z$ with ~$\wt A$:

$$
\begin{CD}
D(\wh M_z)\subset {\mathcal Q}(S_\theta^\circ, \de_\al)@>\wh M_z>>{\mathcal Q}(S_\theta^\circ, \de_\al)\\
@VV\widehat{W_\theta}V @VV\widehat{W_\theta}V  \\
D(\wt A)\subset H_A @>\wt A>>H_A
\end{CD}
$$

\end{itemize}
\end{theorem}

\medskip
In the above result, one can always put $\al=\frac 12$. If $\tht<\frac \pi 2$, then one
can also take $\al=1$.

%
%
%
%

The control
model is the closest one to the original Nagy--Foia\c{s} model, because
it
\margp{he cambiado todo este parrafo}
represents
 $\wt A$
as an operator of multiplication by the independent
variable on
a quotient Hilbert space, see
the Introduction.
However, to the contrary to the original
Nagy--Foia\c{s} setting,
we obtain similarity but not unitary equivalence of
$\wt A$ and the multiplication operator.  In general,
the generalized
characteristic function of $\wt A$
depends on the choice of auxiliary operators $B$ and $C$
and is far from being unique,
see \cite[Section 11]{Yak} for a discussion.

This theorem
\margp{This theorem}
implies that the operator
$\wt A$ possesses a unique weak-star continuous
$H^\infty$-functional calculus in each angle $S_\theta$, if $\theta>\omega$.
Indeed, the unique weak-star continuous $H^\infty$-functional calculus
for $\widehat{M_z}$ is
$$
f(\wh M_z)=\wh M_{f(z)}, \qquad f\in H^{\infty}(S_\theta).
$$
It also implies that there exists a normal dilation of $\wt A$, whose spectrum is
contained in $\partial S_\tht$. We refer to \cite{Fr_Weis_dils}
for dilation results in the context of operators on UMD spaces.

\medskip

The next result provides us with another explicit realization of the operator
$\wt A$, which is called the observation model for $\wt A$, as in \cite{YaPa}.

For $\tht\in(\omega,\pi)$ and $\al$ such that $0<\al<(\pi/2\tht)$, we introduce the model space
$$
{\mathcal H}(\delta_\alpha, S^\circ_\theta)
:=\{v\in E^2(\CC\backslash S_\theta;
H)\,\, ; \,\, \delta_\alpha v\vert_{ \partial S_\theta}\in
E^2({S^\circ_\theta}; H)\}.
$$
In the remainder of this section we assume $\tht$ to be fixed, and
write ${\mathcal H}(\delta_\alpha)$ to denote ${\mathcal
H}(\delta_\alpha, S^\circ_\theta)$.
In fact,
we shall show that, in a certain sense, the above space
does not depend
on $\tht\in(\omega,\pi)$.
\begin{theorem} 
\label{Obs-mod}

Under the above conditions ${\mathcal O}_\theta\vert_{ H_A}$ is an isomorphism of $H_A$ onto
${\mathcal H}(\delta_\alpha)$. Moreover, one has the intertwining formula
${\mathcal O}_\theta D(\wt A)=D(M_z^T), $
$$
{\mathcal O}_\theta \wt Ax=M_z^T {\mathcal O}_\theta x, \qquad
x\in D(\wt A):
$$
\vskip.1cm
$$
\begin{CD}
D(\wt A)\subset H_A@>\wt A>>H_A\\
@VV{\mathcal O}_{\theta}V @VV{\mathcal O}_{\theta}V  \\
D(M^T_z)\subset {\mathcal H}(\delta_\alpha)@>M^T_z>>{\mathcal H}(\delta_\alpha)
\end{CD}
$$
\noindent where $D(M^T_z):=\{f\in {\mathcal H}(\delta_\alpha)\,\, ;\,\, \exists c\in H
\hbox{ such that } M_zf-c\in {\mathcal H}(\delta_\alpha)\}, $ and
\begin{equation}
\label{M_z^T}
M_z^Tf :=M_z f-c\in {\mathcal H}(\delta_\alpha), \qquad f \in D(M^T_z).
\end{equation}
\end{theorem}

\medskip
We call $M^T_z$ the observation model operator
and ${\mathcal H}(\delta_\alpha)$ the observation model space.
Notice that the above definition \eqref{M_z^T} of
$M_z^T$ is correct, because the only constant
$c\in H$ that belongs to
${\mathcal H}(\delta_\alpha)$ is zero.

The formula for the resolvent of $M_z^T$ is given by
\begin{equation}
\label{resolv}
\left((M_z^T-\lambda)^{-1}f\right)(w)={f(w)-f(\lambda)\over
w-\lambda}, \qquad \lambda\not \in \sigma(\wt A).
\end{equation}
\margp{he quitado una frase}

 \bigskip

In order to state our next result, first we reproduce some of the
known equivalent conditions for the existence of the $H^\infty$
calculus for ~$A$.


\begin{theoremFC}
Let $A$ be a one-to-one operator of type $\om$,
where $0<\om<\pi$. Then the following statements are equivalent.
\begin{itemize}
\item[(a)] $A$ admits a bounded $\Hnty(S^\circ_\mu)$-functional calculus for
all $\mu>\om$;

\item[(b)] $A$ admits a bounded $\Hnty(S^\circ_\mu)$-functional calculus for
some $\mu>\om$;

\item[(c)]
$\{A^{is}\vert s\in \RR\}$ is a $C_0$-group and
for any $\mu>\om$, there exists
$c_\mu$ such that $\|A^{is}\|\le c_\mu e^{\mu|s|}$, $s\in\RR$;

\item[(d)]
$H=H_A$, and the norms in these two spaces are equivalent.

\end{itemize}

In the case that $\displaystyle{\omega <{\pi/2}}$, the above conditions are equivalent to

\begin{itemize}

\item[(e)] for any $\tht\in (\om, \pi/2)$, $A$ is similar to a
$\tht$--accretive operator.

\item[(f)] there exists $\tht \in (\om, \pi/2)$
such that $A$ is similar to a $\tht$--accretive operator;
\end{itemize}
\end{theoremFC}

The equivalence of the conditions (a)--(d) was proved by McIntosh in \cite{McInt1}.
The fact that (e) and (f) are equivalent to each of
the conditions (a)--(d), follows from \cite[Theorem 7.3.9]{Haase}.

In \cite[Theorem 1.1]{leMe3}, Le Merdy proves
that $A$ is similar to an $\omega$-accretive operator if and
only if there exists and invertible operator $S\in L(H)$
such that $\|S^{-1}A^{it}S\|\le e^{\omega|t|}$, $t\in\mathbb R$.
On the other hand, as shown in
\cite[Theorem 3]{Simard1999}, the estimate $\|A^{it}\|\le
K\,e^{\omega|t|}$ does not imply the boundedness of
$H^\infty(S^\circ_\theta)$ calculus for $A$
for $\theta=\omega$. Our proofs do not use the apparatus of the complete boundedness.

We refer to Theorem 2.4 in \cite{CDMcY}, Theorem 2.2 in \cite{AuMcIntNahm},
Le Merdy \cite{leMe3,leMe_sim}, the review \cite{Wei_L} and
Chapter 7 of the book \cite{Haase} for additional information.


Our next result shows that this list can be widened.

\begin{theorem} 
\label{eqiv conds}
Let $A$ be a one-to-one operator of type $\omega\in(0,\pi)$.
Then each of the following conditions
is equivalent to conditions (a)--(d) of Theorem FC.
\begin{itemize}

\item[(g)] there exist $\tht \in (\om,  \pi)$ and $\al$ with
$\al\om <{\pi / 2}$ such that the quotient operator $\wh
W_\tht$ induces an isomorphism of the quotient space
$\Q(S_\tht^\circ,\de_\al)$ onto $H$.

\item[(h)] there exist
$\tht \in (\om,  \pi)$
and $\al$ with $\al\om <{\pi / 2}$ such that
the  operator
${\mathcal O}_\tht$ defines an isomorphism of
the space $H$ onto ${\mathcal H}(\delta_\alpha)$.

\item[(i)] for any
$\tht$, $\al$ as above,
the operator $\wh W_\tht$ induces an isomorphism of
$\Q(S_\tht^\circ,\de_\al)$ onto $H$.

\item[(j)] for any
$\tht$, $\al$ as above,
the  operator
${\mathcal O}_\tht$ defines an isomorphism of
$H$ onto ${\mathcal H}(\delta_\alpha)$.


\end{itemize}
\end{theorem}

\section{Analytic extension of inverse characteristic functions}
\label{complex}

\medskip
The class $\Psi(S^\circ_\mu)$ defined in
(\ref{Psi})
can be represented as
the union $\Psi(S^\circ_\mu)=\bigcup_{s>0}\Psi_s(S^\circ_\mu)$, where
$$
\Psi_s(S^\circ_\mu)=
\big\{
f\in {\hbox{Hol}}(S_\mu^\circ):\quad
\|f\|_{\Psi_s} := \sup_{w\in S_\mu^\circ}
\frac {1+\vert w\vert^{2s}} {|w|^s}\, |f(w)|
 < \infty\big\}.
$$
For each fixed $s>0$, $\Psi_s(S^\circ_\mu)$ is a Banach algebra
with respect to the norm $\|\cdot \|_{\Psi_s}$, and is a subalgebra
of $H^\infty(S^\circ_\mu)$.

The present section deals with analytic extensions of rational
expressions involving fractional powers of a sectorial operator.
Concretely, let us assume that $\om<\tht<\pi$ and
$0<\theta\alpha<{\pi }$. The operator
$$
\wt\de_\al(z)
:=
\al( I+2 z^\al (A^\al-z^\al)^{-1})=\al\frac {A^\al+z^\al} {A^\al-z^\al};
\qquad
z\in
\CC
\backslash \big(S_{\omega}\cup(-\infty,0]\big),
$$
is
\margp{Cambiada la formula y el siguiente texto}
the inverse to the operator
$\delta_\alpha(z)$ given in (\ref{delta-al}), whenever
both expressions are well-defined.
Note that $\wt\de_\al$ is defined only for $\al\in(0,\pi/\tht)$,
whereas $\de_\al$ is defined
only for $\al\in(0,\pi/2\tht)$,
that is, the range of possible values
of $\al$ is twice larger in the case of $\wt\de_\al$.
In particular, for any angle $\tht$, we always can take $\al=1$ in $\wt\de_\al$.

Since $A^\alpha$ is $\alpha\tht$--sectorial, we have that
$\wt\de_\al$ is a bounded operator-valued function on $\partial
S_\tht\setminus \{0\}$.
Moreover, if $\al<\pi /2 \tht$, then
$\wt\de_\al(z)=\de_\al^{-1}(z)$, $z\in\partial S_\tht\setminus \{0\}$.
In view of the definition of $\widetilde{ \delta_\alpha}$ we  consider the scalar functions
$$
\ga_{\alpha, z}(w):=
\alpha\;{w^\al+ z^\al \over
w^\al- z^\al},\ \ \qquad z,w \in \CC\setminus(-\infty,0],
$$
(here $\alpha>0$). One has
\beqn
\label{wt de al}
\wt\de_\al(z)=\ga_{\alpha, z}(A), \qquad
z\in\CC\setminus (S_{\mu}\cup(-\infty, 0]),
\neqn
if $\alpha\mu<\pi$.
\medskip

\begin{lemma} 
Let $\mu<\pi $ and $\alpha>0$ such that
$\alpha\mu <\pi$.
\begin{itemize}

\item[(i)] For any fixed $z\in S^\circ_\mu$,
the function $\xi_z$, given by
$$
\xi_z(w):=\ga_{\alpha, z}(w)-\ga_{1, z}(w)
$$
is analytic in $w\in S^\circ_\mu$.

\item[(ii)]
Put $\be=\min(\frac12,\al)$ and
$$\eta_z(w):=\xi_z(w)+(1-\al)\,\frac {w-1}{w+1},\ \  \hbox{ for } z, w\in S^\circ_\mu.$$
Then
the function
$z\mapsto \eta_z$ is analytic from
$S_\mu^\circ$ to the space $\Psi_\be(S^\circ_\mu)$
and satisfies
\beqn
\label{norm in Psi be}
\|\eta_z\|_{\Psi_\be(S^\circ_\mu)}
\le C \big(|z|^\al+|z|^{-\al}\big),
\quad z\in S^\circ_\mu,
\neqn
where the constant $C$ depends only on $\mu$ and $\al$.
\end{itemize}
\end{lemma}

\

\noindent{\it Proof.} (i)
It is straightforward to check that
$$
Res(\ga_{\alpha, z}, z)= 2 z
$$
for any $\alpha$ and  any $z\in S_\mu^\circ$. Since $w=z$ is the
only pole of $\ga_{\alpha, z}(w)$, which is of order one,
assertion (i) follows.

\noindent{(ii)} Take any $\nu$ such that $\mu<\nu<\pi$
and $\alpha \nu<\pi$.
\margp{si no se cumple $\alpha \nu<\pi$, puede suceder que
$\alpha \mu=\alpha \nu$
}
Notice that
$$
|w^\al-z^\al|\ge C_1|z|^\al,
\quad
|w^\al-z^\al|\ge C_1|w|^\al,
$$
for all
$z,w$ such that $z\in S^\circ_\mu$ and
$w\in\partial S_\nu$,
where $C_1>0$ depends only on $\mu$ and $\nu$.
We keep the same notation $C_1$ although it may be different in each
inequality. For $z\in S^\circ_\mu$, $w\in\partial S^\circ_\nu$ with
$|w|\le 1$, it follows that
\begin{eqnarray*}
\big|\ga_{\al,z}(w)-\al\,\frac{w-1}{w+1}\big|
&=&
\al\,\Big|
\frac {2w^\al}{w^\al-z^\al} -\frac {2w}{w+1}
\Big|
\le C_1 (\left|w\right|+\left|{w \over z}\right|^{\al}) \cr\cr&\le& C_1 (1+|z|^{-\al})|w|^\be .
\end{eqnarray*}
Similarly, for
\margp{he desplazado esto aqui}
$z\in S^\circ_\mu$ and $w\in\partial S_\nu$
such that $|w|\ge 1$,
\begin{multline*}
\big|\ga_{\al,z}(w)-\al\,\frac{w-1}{w+1}\big|
=
\al\,\Big|
\frac {2z^\al}{w^\al-z^\al} +\frac 2{w+1}
\Big|
\\
\le C_1 \big( \frac {|z|^\al} {|w|^\al} +\frac 1 {|w|}\big)
\le 2C_1 \frac {|z|^\al+|z|^{-\al} }{|w|^\be}.
\end{multline*}
These two inequalities give
\beqn
\label{ineq al}
\big|\ga_{\al,z}(w)-\al\,\frac{w-1}{w+1}\big|
\le C_1\frac{|w|^\be}{1+|w|^{2\be} }
\;\big( |z|^\al+|z|^{-\al} \big)
\neqn
for $z\in S^\circ_\mu$ and $w\in\partial S_\nu$ where we remind that $C_1$ depends only on
$\mu, \nu$ and $\alpha$.

Note that
$$
C_1( |w|^\be+|w|^{-\be})\le|w^\be+w^{-\be}|\le |w|^\be+|w|^{-\be}, \qquad w\in S_\nu.
$$

Now by applying (\ref{ineq al}) twice (to general $\al$ and to $\al=1$) we find
that there exists $C_2>0$ such that
$$
|(w^\be+w^{-\be})\eta_z(w)| \le C_2
\big( |z|^\al+|z|^{-\al} \big)
$$
for $z\in S^\circ_\mu$ and  $w\in\partial S_\nu$.
By (i), for each fixed $z$ the function
$\eta_z$ is analytic on $S^\circ_\mu$. By the Phragm\'en-Lindel\"of theorem, it
follows that the latter estimate in fact holds for all $z\in S^\circ_\mu$ and $w\in S^\circ_\nu$.
We conclude that for all $z\in S^\circ_\mu$,
$\eta_z$
belongs to
\margp{belongs to}
$\Psi_\be(S^\circ_\nu)$; in particular,
$\eta_z\in \Psi_\be(S^\circ_\mu)$
for $z\in S^\circ_\mu$ and
(\ref{norm in Psi be}) holds true.

Finally, fix $\lambda,z \in S^\circ_\mu$. The function
$$
F(w):=w^{-\beta}(1+w^{2\beta})(\eta_\lambda(w)-\eta_z(w))
$$
is holomorphic in $w\in S_\mu^\circ$ and continuous up to the boundary
$\partial S_\mu^\circ$ (notice that $\lim_{w\in S_\mu^\circ, w\to 0}F(w)=0$),
so that
$$
\Vert \eta_\lambda -\eta_{z}\Vert_{\Psi_\beta(S^\circ_\mu)}
\asymp \sup_{w\in S^\circ_\mu}{\vert1+ w^{2\beta}\vert\over
\vert w\vert^\beta}\vert(\eta_\lambda-\eta_z)(w)\vert
=\sup_{w\in\partial S^\circ_\mu}\vert F(w)\vert
$$
by the Phragm\'en-Lindel\"of theorem.

Writing the function $\eta_\lambda-\eta_z$ as
$$
(\eta_\lambda-\eta_z)(w)=\left[{2\alpha w^\alpha(\lambda^\alpha-z^\alpha)\over(w^\alpha-\lambda^\alpha)
(w^\alpha-z^\alpha)}-{2w(\lambda-z)\over(w-\lambda) (w-z)}\right]
$$
we obtain
\begin{multline*}
\Vert\eta_\lambda-\eta_z\Vert_{\Psi_\beta(S^\circ_\mu)}
\le
\vert\lambda^\al-z^\al\vert\cdot
\sup_{w\in\partial S^\circ_\mu}
\Big|{2\al(1+w^{2\beta})w^\al\over w^\beta(w^\al-\lambda^\al)(w^\al-z^\al)}\Big|
\\
+
\vert\lambda-z\vert
\cdot
\sup_{w\in\partial S^\circ_\mu}
\Big| {2(1+w^{2\beta})w\over w^\beta(w^\al-\lambda^\al)(w^\al-z^\al)}\Big|.
\end{multline*}
From this, and using that $\beta<\hbox{min}\{1/2,\al\}$, it is
readily seen that $\lim_{\lambda\to
z}\Vert\eta_\lambda-\eta_z\Vert_{\Psi_\beta(S^\circ_\mu)}=0$. Thus
the function $\eta\colon z\mapsto\eta_z,\
S^\circ_\mu\rightarrow\Psi_\beta(S^\circ_\mu) $ is continuous.
Then a vector-valued version of the Morera theorem applies to
obtain that $\eta$ is analytic. We have done.
\qed
\begin{proposition}
\label{diff wt de}
For
every  $\al$ such that  $0<\al<\pi/\tht$, the operator-valued function
$z\mapsto\wt\de_\al(z)-\wt\de_1(z)$, defined on
\margp{operator-valued; $\partial S^\circ_\tht\setminus \{0\}$}
$\partial S^\circ_\tht\setminus \{0\}$, continues to a
function on the sector $S^\circ_\tht$ of the class
$\Hnty\big(S^\circ_\tht;  L(H)\big)$.
\end{proposition}

\noindent{\it Proof.} Choose $\be=\min({1\over 2}, \alpha)$.
The map $f\mapsto f(A)$, which
\margp{\quad **}
goes from
$ \Psi_\be(S_\tht^\circ)$ to
$ L(H)$, is
linear and bounded. Hence by (\ref{norm in Psi be}),
$$
\|\eta_z(A)\|\le C
\big(|z|^\al+|z|^{-\al}\big),
\qquad z\in S^\circ_\tht.
$$
Therefore a similar estimate holds for $\xi_z(A)$:
$$
\|\xi_z(A)\|\le C'
\big(|z|^\al+|z|^{-\al}\big),
\qquad z\in S^\circ_\tht.
$$
In particular, the map $z\mapsto \xi_z(A)$ is an analytic
continuation of the map $z\mapsto\wt\de_\al(z)-\wt\de_1(z)$ to the sector $S^\circ_\tht$.
As we noted at the beginning of this section, the operator-valued functions
$\wt\de_\al$ and $\wt\de_1$ are bounded on
$\partial S_\tht\backslash\{0\}$. Now the assertion of the proposition is obtained
by applying Phragm\'en-Lindel\"of theorem to
the scalar functions $z
\mapsto \langle \xi_z(A) h_1, h_2\rangle$,
where $h_1,h_2\in H$.
\qed

\bigskip

\section{Hankel--like operators for sectorial operator}
\label{like}

\medskip
In the following result we introduce a Hankel-like operator on the Hardy--Smirnov class
which is associated to the inverse characteristic function $\wt\delta_\alpha$, and on
which our arguments are based. Thus in principle such an operator depends
on the parameter $\al\in(0,\pi/\tht)$. We shall see as an application
of Proposition \ref{diff wt de} that indeed it is independent of $\al$.

\medskip
\begin{lemma}\label{00}
Let $\tht\in(\omega,\pi)$.

\noindent (1) Define a Hermitian bilinear
\margp{$g(\bar\lambda)$}
pairing by putting
\beqn
\label{corch}
\langle f, g\rangle:={1\over 2\pi i}\int_{\partial S_\theta}\langle
f(\lambda), g(\bar \lambda)\rangle_H \,d\la,
\neqn
for $f\in E^2(S_\theta^\circ; H)$ and $g\in E^2(\CC\backslash S_\theta; H)$.
The spaces  $E^2(S_\theta^\circ; H)$ and $ E^2(\CC\backslash S_\theta; H)$
are dual with respect to this pairing.

\noindent (2) The space $L^2(\partial S_\theta; H)$ splits into the direct sum
$$
L^2(\partial S_\theta; H)= E^2(S_\theta^\circ; H)\oplus
E^2(\CC\backslash S_\theta; H).
$$
This defines parallel continuous projections
$$
P_{int}:  L^2(\partial S_\theta; H)\to  E^2(S_\theta^\circ; H), \qquad
P_{out}:  L^2(\partial S_\theta; H)\to  E^2(\CC\backslash S_\theta;
H),
$$
given by the Cauchy integrals
\begin{eqnarray*}
P_{int}f(z):&=&{1\over 2\pi i}\int_{\partial S_\theta}{f(\xi)\over \xi-z}d\xi,
\qquad z\in S_\theta^\circ,\cr
P_{out}f(z):&=&-{1\over 2\pi i}\int_{\partial S_\theta}{f(\xi)\over \xi-z}d\xi,
\qquad z\in \CC\backslash S_\theta.
\end{eqnarray*}

\noindent (3) Let ${\mathcal H}(\delta_\alpha, S^\circ_\theta)$ be the space defined prior to
Theorem \ref{Obs-mod}. Then ${\mathcal H}(\delta_\alpha, S^\circ_\theta)
=\left(\de_\al E^2(S_\tht^\circ;H)\right)^{\perp}$ for every $\al\in(0, \pi/ 2 \tht)$, where the annihilator is calculated with
respect to the pairing \eqref{corch}.

\noindent (4) For $\al\in(0, \pi/  \tht)$ let us consider the Hankel--like operator
$$
{\mathcal J}_{\wt\de_\al}: E^2(S_\theta^\circ; H)\to  E^2(\CC\backslash
S_\theta; H),
$$
acting by
$$
{\mathcal J}_{\wt\de_\al}(u):=P_{out}(\wt\delta_\al u\vert_ {\partial S_\theta}).
$$
Then Im$\,{\mathcal
J}_{\wt\de_\al}={\mathcal H}(\de_\al)$ and Ker$\,{\mathcal
J}_{\wt\de_\al}$ $=\de_\al E^2(S_\theta^\circ; H)$ for  $\al\in(0, \pi/ 2 \tht)$. Therefore by
factoring ${\mathcal J}_{\wt\de_\al}$ by its kernel we obtain an
isomorphism
$$
\widehat{{\mathcal J}_{\wt\de_\al}}:{\mathcal Q}(S_\theta^\circ, \de_\al)\to
{\mathcal H}(\delta_\alpha, S^\circ_\theta),\ \  \hbox{ whenever }  \al\in(0, \pi/ 2 \tht).
$$
\end{lemma}

\begin{proof}Statements (1)--(3) are contained
in \cite[Propositions 2.1 and 2.2]{Yak}, and statement (4) is
straightforward.
\end{proof}

\medskip



\medskip
Now we show that ${\mathcal J}_{\wt\de_\al}$ is independent of $\al$.

\begin{proposition}\label{Hank}
(1) The Hankel-like operator
$$
{\mathcal J}_{\wt\de_\al}: E^2(S^\circ_\theta; H)\to  E^2(\CC\backslash
S_\theta; H),
$$
does not depend on $\al$ for $\al\in(0,\pi/\tht)$.

(2) The space $\de_\al E^2(S_\theta^\circ; H)$, $0<\al<\pi/(2\tht)$,
does not depend on $\al$.

\end{proposition}

\noindent{\it Proof.} {(1)} By Proposition \ref{diff wt de},
one has
$$
{\mathcal J}_{\wt\de_\alpha}(f)-{\mathcal J}_{\wt\de_{1}}(f)=
P_{out}\big((\wt\delta_{\alpha}-\wt\delta_{1}) f\vert
\partial S_\theta\big)=0.
$$

(2) By  Lemma \ref{00},
$
\de_\al E^2(S_\theta^\circ; H)=Ker({\mathcal J}_{\wt\de_\al}),
$
and the result follows from part (1). \qed

\section{Isomorphism between control and observation spaces}
\label{s aux}

Let $A$ be a sectorial operator of type $\omega\in(0,\pi)$. Our
aim in this section is to prove that the control and observation
spaces associated with $A$ as in Definition \ref{ctrobs} coincide
and have equivalent norms. Let us start with  the
following estimate:

For every $\tht\in(\omega,\pi)$ and $\xi\in\partial S_\theta$,
\begin{multline}
\label{estim norm}
\Big\Vert{A\over (\xi-A)(1+A)}\Big\Vert=
\Big\Vert{\xi(1+A)-(\xi-A)\over
(\xi+1)(\xi-A)(1+A)}\Big\Vert
\\
\le \Big\Vert{\xi\over (\xi+1)(\xi-A)}\Big\Vert+{1\over \vert \xi
+1\vert}\Vert(1+A)^{-1}\Vert\le {C_\theta+ \Vert(1+A)^{-1}\Vert\over
\vert \xi +1\vert},
\end{multline}
where the constant $C_\tht$ comes from the condition on $A$ to be sectorial.

\begin{proposition}
\label{prop1}
For any  $\theta \in (\omega, \pi)$,  the operators
$$
W_\tht \colon E^2(S_\theta^\circ; H)\to H_{-1}\ \hbox { and }\
{\mathcal O}_\tht\colon H_1\to E^2(\CC\backslash S_\theta; H)
$$
are well-defined, linear and bounded.
\end{proposition}

\begin{proof} Let
$u\in E^2(S_\theta^\circ; H)$. Then
$\Vert W_\theta u \Vert_{H_{-1}}
= \Vert \sqrt{A}(1+A)^{-1} \,W_\theta u\Vert_H$ whence
$$
{\pi\over 2}\Vert W_\theta u \Vert_{H_{-1}}^2\le \Vert
u\Vert^2_{E^2(S_\theta^\circ;H)}\int_{\partial S_\theta}\Big\Vert{\sqrt{A}\over
\xi-A}{\sqrt{A}\over 1+A}\Big\Vert^2 \vert d\xi\vert.
$$
By (\ref{estim norm}) it follows that the integral is finite and it proves the
boundedness of $W_\theta: E^2(S_\theta^\circ; H)\to H_{-1}$. Now, for
$x\in H_1$, note that
\begin{multline*}
\int_{\partial S_\theta}\Vert
{\mathcal O}_\tht x(z)\Vert^2_H \vert dz\vert= \int_{\partial S_\theta}\Vert
\sqrt{A}(z-A)^{-1}x\Vert^2_H \vert dz\vert
\\
\le\int_{\partial
S_\theta}\Big\Vert {A \over (z-A)(1+A)}\Big\Vert^2
\;
\big\Vert {1+A\over
\sqrt{A}} \,x\big\Vert^2_H \vert \,dz\vert
\le C\Vert x\Vert_1^2,
\end{multline*}
\medskip
where the last inequality is obtained again from (\ref{estim norm}). Hence
${\mathcal O}_\tht\colon H_1\to E^2(\CC\backslash S_\theta; H)$ is well defined and bounded. The  proposition  is proved.
\end{proof}

Recall Definition \ref{ctrobs}: the control space for $A$ is
$\Hctr:=\hbox{Im}\, W_\tht\subset H_{-1}$ endowed with the norm
$\Vert x\Vert_{A,\tht,\text{ctr}}
:=\min\{ \Vert u \Vert_{E^2(S_\theta^\circ;H)}\, ; \,\, x=W_\theta(u)\}$,
and the observation space for $A$ is defined as
$\Hobs:={\mathcal O}_\tht^{-1}\big(E^2(\CC\backslash{S_\theta})\big)\subset H_{-1}$
with the norm
$\Vert x\Vert_{A,\tht,\text{obs}}
:= \Vert {\mathcal O}_\theta x\Vert_{E^2(\CC\backslash{S_\theta}; H)}$.
In the next proposition,
\margp{cambiado}
we collect several basic facts about the spaces $\Hctr$ and $\Hobs$.
The arrows \lq\lq $\hookrightarrow$" mean continuous inclusions. The symbol
$\widehat W_\tht$ denotes the quotient mapping
$E^2(S_\theta^\circ;H)/ \hbox{Ker}\, W_\tht\to \hbox{Im}\,W_\tht$.

\begin{proposition}\label{inclusions}
\noindent (1) The mapping
$$
\widehat W_\tht\colon E^2(S_\theta^\circ;H)/ \hbox{Ker } W_\tht\to\Hctr
$$
is an isometric isomorphism, and $H_1\hookrightarrow \Hctr$.

\noindent (2) The mapping
$$
{\mathcal O}_\theta\colon\Hobs\to E^2(\CC\backslash{S_\theta}; H)
$$
is an isometry, and $H_1\hookrightarrow \Hobs \hookrightarrow H_{-1}$.

\noindent (3) The space  $\Hobs$ is complete.
\end{proposition}

\begin{proof} (1) The norm of $x=W_\tht(u)$ in $\Hctr$ is exactly the quotient norm of
$u+\Ker W_\tht$ in
$E^2(S_\theta^\circ;H)/ \Ker  W_\tht$, so $\widehat W_\tht$ is an isometric isomorphism.

For any $\lambda \in \CC\backslash
S_\theta$ and any $x\in H$ the rational function
$u_{\lambda,x}(z):=(\lambda -z)^{-1}x$ belongs to $E^2(S_\theta^\circ; H)$.
Then the Dunford-Schwartz calculus gives us
$$
W_\theta(u_{\lambda,x})=2\sqrt{A}(\lambda-A)^{-1}x\in \Hctr.
$$
On the other hand, for fixed
$\lambda$, the vectors $y:=2\sqrt{A}(\lambda-A)^{-1}x$ range over the
whole space $H_1$ if $x$ runs over $H$.
Hence $H_1\subset \Hctr$.
\margp{cambiado}
The continuity of this inclusion follows from the estimate
$$
\Vert y\Vert_{A, \theta, ctr}\le \Vert u_{\lambda,x}\Vert_{E^2}
\le C_\lambda\Vert x\Vert_H
\le  C_\lambda^1\Vert y\Vert_1,
$$
where $C_\lambda$, $C^1_\lambda$
are constants depending on $\lambda$.


\medskip
(2) That ${\mathcal O}_\theta\colon\Hobs\to E^2(\CC\backslash{S_\theta}; H)$ is an isometry is clear from the definition of
$\Hobs$, and then $H_1\hookrightarrow \Hobs$ is a straightforward consequence of this isometry and Proposition \ref{prop1}.

Now for every $\lambda\in \CC\setminus S_\theta$
%
and $x\in\Hobs$ we have that
$$
\sqrt {A}(1+A)^{-1}x=\big((\lambda+1)(1+A)^{-1}-1\big)\cO_\tht x(\lambda)
$$
 whence
$\Vert x\Vert_{-1}\le (a\vert\lambda\vert+b)\Vert {\mathcal O}_\theta x(\lambda)\Vert_H$,
and therefore we obtain
$c\Vert x\Vert_{-1}^2\le\Vert {\mathcal O}_\theta x\Vert_{E^2(\CC\backslash{S_\theta}; H)}=\Vert x\Vert_{A,\tht, \text{obs}}^2$
where $a,b, c$ are positive constants.

\medskip
(3)  Let $(x_n)\subset \Hobs$
be a Cauchy sequence in $\Hobs$.
By  (2) above, there exists $x \in H_{-1}$ and $v\in E^2(\CC\backslash S_\theta; H)$
such that $x_n\to x$ in $H_{-1}$ and
\margp{at the same time}
at the same time
${\mathcal O}_\theta x_n\to v$ in $E^2(\CC\backslash S_\theta;H)$.
Since the linear map $y\mapsto {\mathcal O}_\theta y(\lambda)$
is continuous in $y\in H_{-1}$ for each fixed $\lambda\in \CC\backslash S_\theta $, we conclude that
${\mathcal O}_\theta x=v, $ hence $x_n\to x$ in
$\Hobs$.
\end{proof}

\medskip
As it has been seen in Proposition \ref{prop1} (1), the composition mapping
$$
E^2(S_\theta^\circ; H)
\overset
{W_\theta} {\longrightarrow}
H_{-1}
\overset
{{\mathcal O}_\theta} {\longrightarrow}
\hbox{Hol}(\CC\backslash
S_\theta; H)
$$
is well defined. Now, via the Hankel-like operator, we  prove that
in fact the range of $\cO_\tht W_\theta$ lies in
$E^2(\CC\backslash S_\theta; H)$.

\begin{lemma}
\label{37}
For any $\tht\in(\omega, \pi)$ we have $\cO_\tht W_\theta = -{\mathcal J}_{\wt\de_1}$.
In particular, ${\mathcal O}_\theta W_\tht$ is
bounded from $E^2(S_\theta^\circ; H)$ to
$E^2(\CC\backslash S_\theta; H)$.
\end{lemma}

\noindent {\it Proof.} Take $\lambda \in
\CC\backslash S^\circ_\theta,$ and $u\in E^2(S^\circ_\theta; H)$.
We have
\begin{eqnarray*}
&\quad&({\mathcal O}_\theta W_\theta u)(\lambda)=
{1\over 2\pi i}\int_{\partial S_\theta}{2A u(z)\over (\lambda-A)(z-A)}\,dz\cr
&=&{1\over 2\pi i}\int_{\partial S_\theta}{2A u(z)\over (\lambda-A)(z-\lambda)}\,dz
+\,{1\over 2\pi i}\int_{\partial S_\theta}\big(I +\frac {A+z} {A-z}\big)\;
\frac{u(z)}{z-\lambda}\,dz\cr
&=&{1\over 2\pi i}{A+\lambda\over \lambda-A}
\int_{\partial S_\theta}{u(z)\over z-\lambda}\,dz+
{1\over 2\pi i}\int_{\partial S_\theta}{A+z \over A-z}{u(z)\over z-\lambda}\,dz\cr
&=&-{A+\lambda\over \lambda-A}(P_{out}u)(\lambda)
-{\mathcal J}_{\wt\de_1}(u)(\lambda)=-{\mathcal J}_{\wt\de_1}(u)(\lambda),\cr
\end{eqnarray*}

as we wanted to show.\qed

\medskip
As a first application of the above result we get
the following assertion.

\begin{proposition} \label{19} There is a continuous embedding
$\Hctr\hookrightarrow \Hobs$.
\end{proposition}

\noindent {\it Proof.}
Take any $x\in\Hctr$. Then $x= W_\theta u$ for some function $u \in E^2(S_\theta^\circ; H)$.
By Lemma \ref{37}, ${\mathcal O}_\theta W_\theta u\in  E^2(\CC\backslash
S_\theta; H)$, $x=W_\theta u \in \Hobs$,  and
$$
\Vert x\Vert_{A,\tht,\text{obs}} = \Vert {\mathcal O}_\theta
x\Vert_{E^2(\CC\backslash{S_\theta}; H)}=\Vert {\mathcal O}_\theta W_\theta u\Vert_{E^2(\CC\backslash
S_\theta; H)}\le C_\theta \Vert u\Vert_{E^2(S_\theta^\circ; H)}.
$$
Therefore $\Vert x\Vert_{A,\tht,\text{obs}}\le C_\theta \Vert x\Vert_{A,\tht,\text{ctr}}$
and $\Hctr\hookrightarrow \Hobs$.
 \qed

\medskip
Now our aim is to prove the (continuous) reverse inclusion. We shall use the following approximation procedure.
For any $x\in H_{-1}$, we define
\begin{equation}
\label{eq6}
x_\varepsilon:={(1-\varepsilon^2)A\over
(A+\varepsilon)(1+\varepsilon A)}x=
{\varepsilon^{-1}\over
A+\varepsilon^{-1}}x-{\varepsilon\over A+\varepsilon}x, \qquad 0<\varepsilon<1.
\end{equation}
If $v={\mathcal O}_\theta x$, then ${\mathcal O}_\theta
x_{\varepsilon}=v_{(\varepsilon)}, $ where
$$
v_{(\varepsilon)}(z):={\varepsilon}^{-1}{v(z)-v(-\varepsilon^{-1})\over
z+\varepsilon^{-1}}-\varepsilon{v(z)-v(-\varepsilon{})\over
z+\varepsilon}.
$$
This follows from  definitions of $\cO_\theta $ and $x_\varepsilon$.

\begin{lemma} \label{6}(1) For any $x\in H$, one has
\margp{one has}
$x_\varepsilon \to x$
in $H$ as $\varepsilon \to 0^+$.

\noindent (2) For any $v\in E^2(\CC\backslash S_\theta; H)$,
$v_{(\varepsilon)}\to v$ in $E^2(\CC \backslash S_\theta; H)$ as
$\varepsilon \to 0^+$.

\noindent (3) $H_1$ is dense in $\Hobs$.
\end{lemma}

\noindent {\it Proof.}
Part
(1) follows from (\ref{eq6}) and \cite[Proposition 2.1.1]{Haase}.
Part (2)
can be deduced  in the same way as part (1). Indeed, the
multiplication operator $M_z$ on
$E^2(S_\theta^\circ; H)$ given by
$$
M_z(u)(z)=zu(z), \qquad z\in S_\theta^\circ, \qquad u \in {\mathcal D}(M_z),
$$
is the adjoint to $M_z^T$ on
$E^2(\CC\backslash S_\theta; H)$, see for instance \cite[Formula (2.7)]{Yak}.
Hence both are
$\theta$-sectorial and since
$$
v_\varepsilon={(1-\varepsilon^2)M_z^T\over
(M^T_z+\varepsilon)(1+\varepsilon M^T_z)}\;v,
$$
(\cite[Proposition 1.2]{Yak}), the assertion follows again
from \cite[Proposition 2.1.1]{Haase}. (One can also deduce it from straightforward
direct estimates.)

Finally take any $x\in \Hobs$. By (2) of Proposition
\ref{inclusions}, $\Hobs\hookrightarrow H_{-1}$. Note that
$x_\varepsilon \in H_1$ for any $\varepsilon \in (0,1)$. By part
(2) of
\margp{cambiado}
this
lemma, we have that $x_\varepsilon \to x$ in $\Hobs$ as
$\varepsilon \to 0^+$. This shows part (3). \qed

\begin{lemma}\label{sabina}
${\mathcal O}_\theta \Hobs \subset {\mathcal H}(\delta_\alpha)$ for all $\al\in(0,\pi/2\tht)$.
\end{lemma}

\noindent {\it Proof.} It suffices to check that
${\mathcal O}_\theta H_1\subset {\mathcal H}(\delta_\alpha)$,
because
$\cO_\theta:\Hobs\to E^2(\CC\backslash S_\theta; H)$
is an
isometry, $H_1$ is dense in $\Hobs$ and ${\mathcal
H}(\delta_\alpha)$ is a closed subspace of $E^2(\CC\backslash
S_\theta; H)$. Then, by (1) of Proposition \ref{inclusions}, it is
enough again to show that ${\mathcal O}_\theta \Hctr\subset
{\mathcal H}(\delta_\alpha)$.  Take $x\in \Hctr$,  $x=W_\theta(u)$
with $u \in E^2(S_\theta^\circ;H)$. We apply Lemma \ref{37} to
obtain  ${\mathcal O}_\theta(x)=-{\mathcal J}_{\tilde
\delta_1}(u).$ Finally note that ${\mathcal J}_{\tilde
\delta_1}(u) \in  {\mathcal H}(\delta_\alpha)$ by Proposition
\ref{Hank} and Lemma \ref{00} (4). \qed

\begin{proposition}
\label{122}
For every $\theta \in (\omega, \pi)$,
$\Hobs=\Hctr$ with equivalent norms.
\end{proposition}

\noindent {\it Proof.} It has been shown in Proposition \ref{19}
that $\Hctr\hookrightarrow \Hobs$. To prove the reverse inclusion,
take any $x\in \Hobs$. By Lemma \ref{sabina}, ${\mathcal O}_\theta
x\in {\mathcal H}(\delta_\alpha)
= {\mathcal J}_{\tilde\delta_\alpha} (E( S_\theta^\circ; H))$.
Let consider the quotient map
$$
\widehat{{\mathcal J}_{\tilde\delta_\alpha}}:
{\mathcal Q}(S_\theta^\circ, \delta_\alpha)\to E^2(\CC \backslash
S_\theta;H),
$$
 and  put
$\widehat{w}=
\widehat{{\mathcal J}_{\tilde\delta_\alpha}}^{-1}({\mathcal O}_\theta
x)\in {\mathcal Q}(S_\theta^\circ, \delta_\alpha)$.
Note that $y:=W_\theta w\in  H_{-1}$ by Proposition \ref{prop1}.
Then
$$
{\mathcal O}_\theta y={\mathcal O}_\theta W_\theta w
= -{\mathcal J}_{\tilde\delta_1} w= -{\mathcal J}_{\tilde\delta_\al} w
=-{\mathcal O}_\theta x.
$$
Since $x\in H_{-1}$ the injectivity of $\cO_\tht$ on $H_{-1}$ implies that $x=-y \in \Hctr$.\qed

\section{Proofs of main results}
\label{Proofs main res}

%
%

%
%
%
%

%
%




 Theorem \ref{thm log} is not necessarily associated with functional models, so we give here a self-contained proof of it.

\noindent{\it Proof of Theorem \ref{thm log}.}
We are using the Riesz-Dunford calculus. Firstly, we check that
$$
\Vert\Lambda_k(A)^{-r}x\Vert_A\le C\Vert x \Vert, \qquad x\in H, k\in\ZZ\backslash\{0\}.
$$

To do this, take  $\psi \in \Psi(S^\circ_\tht)$.
By applying (\ref{normMC}) to $\Lambda_k(A)^{-r}x$, the Riesz-Dunford formula
for $\psi(tA)\Lambda_k(A)^{-r}x$ and (\ref{sect_cond}), we get
\begin{eqnarray*} &\, &\Vert
\Lambda_k(A)^{-r}x\Vert_A^2\cr&\quad&\le C\Vert x\Vert^2
\sum_{\tau=\pm 1}\int_0^\infty\left(\int_0^\infty\vert \psi(tre^{i\tau
\theta})\vert( 1+ \vert \log({r})\vert)^{-r}
{dr\over r}\right)^2{dt\over t}\cr
  \cr &\quad&
  = C\Vert
    x\Vert^2
    \sum_{\tau=\pm 1}\int_0^\infty\left(\int_0^\infty\vert \psi(se^{i\tau
    \theta})\vert( 1+ \vert \log({s\over t})\vert)^{-r}{ds\over
    s}\right)^2{dt\over t}\cr
    &\quad&\le C\Vert x\Vert^2\sum_{\tau=\pm 1}\left(\int_0^\infty\vert \psi(se^{i\tau \theta})\vert
    \left(\int_0^\infty ( 1+ \vert
    \log(u)\vert)^{-2r}{du\over u}\right)^{1\over 2}{ds\over s}\right)^2
    \end{eqnarray*}
where we have applied
the change of variable $u=s/t$ in the inner integral
and
the Minkowsky inequality
$$
\Big\Vert \int_0^\infty
\vert f(\cdot, s)\vert \,{ds\over s} \;\Big\Vert_2^2\le  \big(\int_0^\infty\Vert
f(\cdot, s) \Vert_2\,{ds\over s}\big)^2.
$$
It follows that
    $$
    \Vert \Lambda_k(A)^{-r}x\Vert_A\le C\sum_{\tau=\pm 1}\Vert
    x\Vert\int_0^\infty\vert \psi(se^{i\tau \theta})\vert {ds\over
    s}=C_1\Vert x\Vert.
    $$

\margp{he cambiado $C$ por $C_1$}
In order to prove the first inequality
in \eqref{Lambda},
note that the adjoint operator $A^{*}$ of $A$ on $H$ is also sectorial of type
$\omega$. Then by applying the preceding estimate to $A^{*}$, we obtain that the operator
$\Lambda_m(A^{*})^{-r}\colon H\to H_{A^{*}}$ is bounded for all $m\in\ZZ\backslash\{0\}$.

Now there is a (natural) duality $(H_A)^{*}=H_{A^{*}}$, see \cite[Theorem 2.1]{AuMcIntNahm},
and the adjoint operator of
$\Lambda_m(A^{*})^{-r}\colon H\to H_{A^{*}}$ is $\Lambda^{*}_m(A)^{-r}\colon H_A\to H$ where
$\Lambda^{*}_m(z):=\overline{\Lambda_m(\overline z)}$, $z\in\CC\setminus(-\infty,0]$.

For every $y\in H_A$, $x\in H$,
\begin{eqnarray*}
  &\quad& \langle(\Lambda_m(A^{*})^{-r})^{*}y,x\rangle_H=\langle y,\Lambda_m(A^{*})^{-r}x\rangle_{(H_{A^{*}})^{*},H_A^{*}}
   \cr \cr
   &\quad&\qquad\qquad:=\int_0^\infty\langle\psi_t(A)y,\psi^{*}_t(A^{*})\Lambda_m(A^{*})^{-r}x\rangle_H{dt\over t}\cr
   &\quad&\qquad\qquad=\int_0^\infty\langle\psi_t^2(A)\Lambda^{*}_m(A)^{-r}y,x\rangle_H{dt\over t}
   =\langle\Lambda^{*}_m(A)^{-r}y,x\rangle.
  \end{eqnarray*}

   \noindent Here we have taken $\psi$ in $\Psi(S_\tht^\circ)$ such that $\psi^2\in\Psi(S_\tht^\circ)$,
   $\int_0^\infty\psi^2(\tau){d\tau\over \tau}=1$. Finally note that $\Lambda^{*}_{-k}(z)=\Lambda_k(z)$ for a fixed $k\in\ZZ$, from which it follows that $\Lambda_k(A)^{-r}\colon H_A\to H$ is bounded as we wanted to show.

To conclude the proof, let us see that $\Lambda_k(A)^{-r}H$ is dense in $H_A$.
Choose a sequence  $(f_n)_n$ such that $\ (\Lambda_k^r f_n)_n\subset \Psi(S^\circ_\theta)$ for $r>1/ 2$, and
$\lim_n f_n(A)x=x$ in $H_A$ for every $x\in H_A$ (for the existence of such a sequence,
see for example \cite[p. 165]{CDMcY} and \cite[p. 170]{McIYa}).
Note that if $x\in H_c$ (see Introduction) then
$\Lambda_k^r(A) f_n(A)x\in H_c$ by (\ref{FC_A}) and
$$
\Lambda_k^{-r}(A)\left(\Lambda_k^r(A)f_n(A)x\right)=f_n(A)x\rightarrow x.
$$
Now it suffices to apply the density of $H_c$ in $H_A$.
\qed

\noindent{\it Proof of  Theorem \ref{main}.} We begin by proving that $\Hobs=H_A$ with
equivalent norms. To see this, let first show that for any $\tht_1,\tht_2\in (\om,\pi]$ there
exists a constant $K=K(\tht_1,\tht_2)>0$ such that
\beqn
\label{roots}
\int_0^\infty \|\,\root\of A\big(re^{i\tht_1}-A\big)^{-1}x\|^2\,dr
\le K
\int_0^\infty \|\,\root\of A\big(re^{i\tht_2}-A\big)^{-1}x\|^2\,dr
\neqn
for all
$x\in H_{-1}$.
Denote by $\Gamma_\phi:=\{re^{i\phi}\,\,; r>0\}$
with $\phi\in (-\pi, \pi]$. Put $\tht=\min(\tht_1,\tht_2)$. The estimate
\begin{multline*}
\| (re^{i\tht_1}-A)^{-1}
\,\root\of A x
\|
\le
\big\|
\frac{re^{i\tht_2}-A}
{re^{i\tht_1}-A}\,
(re^{i\tht_2}-A)^{-1}
\,\root\of A\, x
\big\| \\
\le
\big( C_\tht|e^{i\tht_1}-e^{i\tht_2}|+1 \big)
\| (re^{i\tht_2}-A)^{-1}
\,\root\of A\, x
\|,
\end{multline*}
follows from  (\ref{sect_cond}).
Hence for every two angles $\tht_1, \tht_2\in (\om,\pi]$, the
quantities
$\|{\mathcal O} x\|_{L^2({\Gamma_{\tht_1}})}$ and
$\|{\mathcal O} x\|_{L^2(\Gamma_{\tht_2})}$
are comparable for all $x\in H_{-1}$. This  proves \eqref{roots}.

Now it follows from
\cite[Theorem 2.7]{vanWint71}
that if
${\mathcal O} x|_{\partial S_{\tht_0}}\in L^2\big(\partial S_{\tht_0}\big)$
for
\margp{cambiado}
some $\tht_0\in (\om,\pi)$, then
${\mathcal O} x|_{\CC\setminus  S_\tht}\in
E^2\big(\CC\setminus  S_\tht; H\big)$
for any $\tht\in (\om,\pi)$
(the general case reduces easily to
the case $\pi/2=\tht<\tht_0$).
On the other hand, it is readily seen that
\begin{equation}
\label{norms}
\int_0^\infty \big\| \psi_t(A) x \big\|^2 \, \frac
{dt} t =\| {\mathcal O} x \|^2_{L^2(\Ga_\pi)}
\end{equation}
for all $x\in H_{-1}$ if one chooses
$\psi(z)=\root\of z(1+z)^{-1}\in \Psi(S_\om^\circ)$. Thus $H^{obs}_{A,\tht}=H_A$
with equivalent norms,
for  any $\tht\in (\om,\pi)$.

Also, by  Proposition \ref{122}, $\Hobs=\Hctr$ with equivalent norms. Hence we conclude
that $\Hobs=\Hctr=H_A$ for any $\theta \in (\omega, \pi)$ with
mutually equivalent norms.\qed

\bigskip
\noindent{\it Proof of Theorem \ref{Ctr-mod}.}
Fix any $\alpha$ such that $0<\alpha<\pi/2\tht$.
By Lemma \ref{37}, Proposition \ref{Hank} (1) and
Lemma \ref{00} (4),
$$
\Ker {\mathcal O}_\tht {\mathcal W}_\tht =
\Ker {\mathcal J}_{\wt\de_1} =
\Ker {\mathcal J}_{\wt\de_\al}u
= \de_\al E^2(S^\circ_\tht; H)
$$
for $0<\al<\pi/2\theta$. Further, $\hbox{Ker}\,{\mathcal O}_\theta\, W_\theta
=\hbox{Ker}\, W_\theta$ since ${\mathcal O}_\theta$ is injective and so (i) is proved.

Part (ii) is a consequence of Proposition \ref{inclusions} (1) and Theorem \ref{main}.

Note that
the control operator $W_\theta$ intertwines the resolvent of
the operator $A$ with the resolvent of the model
operator $M^T_z$:
$$
(A-\lambda)^{-1}
W_\tht=W_\theta\,M^T_ {(z-\lambda)^{-1}},
\quad \la\in\CC\backslash S_\tht,
$$
 see \cite[formula (5.1)]{YaPa} and for any rational scalar function $q\in H^\infty(S_\theta^\circ)$, we have
$$
W_\theta(qf)=q(A)W_\theta(f), \qquad f \in E^2(S_\theta^\circ;H).
$$
In particular $M_z\big(D(M_z)\cap(\Ker W_\theta )\big)\subset \Ker W_\theta$,
and therefore $\wh M_z$ is well-defined.  Then the operator
 $\wh M_z$ is closed, densely defined and
 $$
 \widetilde A\,\wh W_\theta= \wh W_\theta\wh M_z,
 $$
see a similar proof in \cite[Theorem 5.6]{YaPa}. This shows part
(iii) and the proof is concluded. \qed

\bigskip
\noindent{\it Proof of  Theorem \ref{Obs-mod}}. Since $\Hobs=H_A$ then $H_A\hookrightarrow H_{-1}$
by Proposition \ref{inclusions}. Then ${\mathcal O}_\theta\vert_{H_A} $ is one-to-one and we have
${\mathcal O}_\theta\vert_{H_A} :H_A\hookrightarrow {\mathcal H}(\delta_\alpha)$.
Take $v \in  {\mathcal H}(\delta_\alpha)$.
By  Lemma \ref{00} (4) there exists $u\in  E^2(S^\circ_\theta; H) $ such that
${\mathcal J}_{\widetilde \delta_\alpha}(u)=-v$ with $0<\alpha<\pi/2\tht$.
Set $x=W_\theta(u) \in H_A$. By Lemmas \ref{37} and \ref{Hank} (1), we obtain that
$$
{\mathcal O}_\theta\vert_{H_A}(x)=
{\mathcal O}_\theta(W_\theta(u))=-{\mathcal J}_{\widetilde \delta_1}(u)=
-{\mathcal J}_{\widetilde \delta_\alpha}(u)=v,
$$
and we conclude that ${\mathcal O}_\theta\vert_{H_A}$ is an isomorphism.

Now we apply the Hilbert identity and the equality (\ref{resolv}) to show  that
$$
{\mathcal O}_\theta((\lambda-\widetilde A)^{-1} x)=(\lambda-M_z^T)^{-1}{\mathcal O}_\theta(x), \qquad  x\in H_A,
$$
for $\lambda \not \in \sigma( \widetilde A) $ whence
$$
\qquad \qquad \qquad \qquad \qquad
{\mathcal O}_\theta \wt Ax=M_z^T {\mathcal O}_\theta x, \qquad x\in D(\wt A).
\hfill \;\qquad \qquad
\qed
$$

\noindent{\it Proof of  Theorem \ref{eqiv conds}.}
The
equivalence of (d), (g) and (i) is a direct consequence of
Theorem \ref{Ctr-mod} (ii) and  Theorem \ref{main}. The
equivalence of (d), (h) and (j) is from Theorem \ref{Obs-mod}.
\qed

\section{Comments and final remarks}
\label{Sec_final}


\noindent  {\bf a.} {\it An alternative proof of  Theorem \ref{thm log}}. The functional model
allows us to give the following argument: To prove the inclusion
and the second inequality, note that the function $\varphi_{r,x}$ given by
$$
\varphi_{r, x}(z):= \frac {\Lambda_k(z)^{-r}} {\root\of z}x, \qquad x\in H, \quad k\in \ZZ\backslash\{0\},
$$
belongs to $E^2(S_\theta^\circ; H)$ for all $\theta>\om$ and all $r>\frac 12$. Moreover,  for any $x\in H$,
\begin{equation}
\label{W1}
W_\theta (\varphi_{r, x})
=\Lambda_0(A)^{-r}x.
\end{equation}
To show this, observe that both parts depend continuously on $x\in
H$. So it suffices to check \eqref{W1} for vectors $x$ of the
form $x=T_Ax_1$ with $x_1\in H$. For these vectors, the equality follows
from the $\Psi(S^\circ_\theta)$-functional calculus. Now take any
element $h\in H$ of the form $h=\Lambda_k(A)^{-r}x$, $x\in H$.
By \eqref{W1} and
\margp{Proposition}
Proposition \ref{122},
we have that
$h\in \Hobs=\Hctr$ and, by  Theorem \ref{main}, that $h\in H_A$. Then by
Proposition \ref{inclusions} (1) we obtain the inequality
$$
\Vert \Lambda_k(A)^{-r}x\Vert_{A}
\le C\Vert W_\theta(\varphi_{r, x})\Vert_{A,\tht,ctr}\le C'\Vert \varphi_{r,x}\Vert_{E^2(S^\circ_\theta; H)}\le C_r \Vert x\Vert.
$$

The rest of the proof follows the same lines as the proof of Theorem
\ref{thm log}, which was given in Section \ref{Proofs main res}.
\qed

\medskip
\noindent  {\bf b.} {\it Admissibility}.
\margp{Cambiado}
Let us assume that $-A$ is the infinitesimal generator of a
bounded $C_0$-semigroup $(e^{-tA})_{t>0}$ on $H$. Let $C$ be  an observation
operator $C:\mathcal{D}(A)\to Y$, for some Hilbert space $Y$, which is continuous with respect to the graph norm
of $\mathcal{D}(A)$. Then $C$ is called admissible if it satisfies the estimate
\begin{equation}
\label{obs inq}
\int_0^\infty \|Ce^{-tA}x\|^2\, dt\le K\|x\|^2, \qquad x\in \mathcal{D}(A),
\end{equation}
for some positive constant $K$. Admissible (and exact) observation operators are important
in linear Control Theory, in particular in the
linear quadratic optimization problem, see
\cite{PartPott, Staffans, Mikkola} and references therein.

In \cite{Lemerdy}, see also \cite[p. 204]{leMeW}, admissible
operators have been studied in terms of the admissibility of the
operator $\sqrt{A}$ (in this case $Y=H$), for bounded analytic
semigroups $(e^{-tA})_{t>0}$ or equivalently when $A$ is sectorial
of type $\omega<\pi/2$ (see \cite[Proposition 2.2]{leMeW}). In
particular $\sqrt{A}$ is admissible
\margp{if $A$ has}
if $A$ has
a $H^\infty$
functional calculus. Here we obtain the following corollary of
Theorem \ref{thm log}.

\begin{corollary}\label{adm log}
Let $A$ be a sectorial operator  such that $-A$ generates a
$C_0$-semigroup $(e^{-tA})_{t>0}$. Then the operator
$C:=\Lambda_k(A)^{-r}\sqrt A$ is admissible for $A$ whenever
$r>1/2$ and for all $k\in \ZZ\backslash\{0\}$.
\end{corollary}

\begin{proof} This is easy. For every $x\in\mathcal D(A)$,
$$
\int_0^\infty \|Ce^{-tA}x\|^2\, dt=\int_0^\infty \|\sqrt{tA}\ e^{-tA}\Lambda_k(A)^{-r}x\|^2\, dt
$$
$$
=\|\Lambda_k(A)^{-r}x\|_A^2\le K\|x\|^2
$$
by Theorem \ref{thm log}.
\end{proof}

\bigskip

\noindent  {\bf c. }{\it On $\om$-accretive operators}.
We recall that a closed operator
$T$ on a Hilbert space $H$ is called
$\om$-accretive if its numerical range
$\big\{\langle Tx,x\rangle; x\in H, \|x\|\le 1\big\}$
is contained in the closed sector $S_\om$.
\begin{proposition} 
\label{Prop:spectr}
For any $x\in H_A$, consider $Ax$ as an element
of $H_{-3}$. Define a (possibly unbounded) operator $\wt A$ on
$H_A$ by
$$
D(\wt A):=\{x\in H_A\,\, ; \,\, Ax\in H_A\},
$$
and $\wt Ax=Ax, $  $x\in D(\wt A)$. Then the following holds.
\begin{enumerate}

\item[(a)] $\wt A$ is similar to an $\omega$-accretive
operator;

\item[(b)]
$\wt A$ has trivial kernel, and $\sigma(A)=\sigma(\tilde A)$.
\end{enumerate}
\end{proposition}

We notice that  (a) follows from \cite[Theorem 7.3.9]{Haase}.
We will see that this fact also follows immediately from our
main results.

\begin{proof}

Fix some angle $\tht\in(\om,\pi)$ and some $\al$
as in Theorem \ref{Ctr-mod}.
The spectrum $\spec \de_\al$ of the $L(H)$-valued
analytic function
$\de_\al$ on $S^\circ_\tht$ is defined as the set of all points
$\la\in S_\tht$
such that $\de_\al^{-1}\notin H^\infty\big(S^\circ_\tht\cap \cW\big)$
for any neighborhood $\cW$
of $\la$.

By Theorem \ref{Ctr-mod}, $\wt A$ is similar to the
quotient multiplication operator $\wh M_z$ on
${\mathcal Q}(S_\theta^{\circ}, \de_\al)$.
It is immediate that this operator is
$\om$-accretive, which gives (a).

The kernel of $\wh M_z$  is zero.
Indeed, if $\wh M_z \rho=0$, $\rho=[r]$,
$r\in E^2(S_\theta^{\circ}, H)$,
then
$zr(z)=\de_\al(z) h(z)$ for some $h\in E^2(S_\theta^{\circ}, H)$.
It follows from the properties of $\de_\al$ that
$h(z)=zh_1(z)$, $h_1\in E^2(S_\theta^{\circ}, H)$,
and therefore $\rho=[\de_\al\cdot h_1]=0$.
(One can also make use of the fact that the model operator like the one considered here
is an analogue of a completely nonunitary contraction in the Nagy--Foia\c{s}
theory. Hence it cannot have a point spectrum on
$\partial S_\theta$.)

It also follows from well-known results that the spectrum of $\wh M_z$
coincides with the set $\spec\de_\alpha$, see
\cite[VI.4.1]{SzNF} for the case of the disc
(which transplants easily to any simply connected Jordan domain)
or \cite[Proposition 2.3]{Yak}.
So it only remains to prove that $\spec \de_\al=\sigma(A)$.
Let us assume first that $\om<\pi/2$, then we can put
$\al=1$ and take $\tht<\pi/2$.

Let $\la\in S^\circ_\tht$, $\la\ne0$.
Then
the following properties are equivalent:
(i) $\la\notin\sigma(A)$;
(ii) $\de_1(\la)$ is invertible;
(iii) $\de_1(z)$ is invertible, with a uniform estimate on the norm of the
inverse, for $z$ in some neighborhood of $\la$.
It follows that
$$
\sigma(A)\setminus \{0\}
=
\spec \de_1\setminus \{0\}.
$$
Now let us consider the remaining case when $\la=0$.
If $0\notin\sigma(A)$, then obviously,
$\de_1^{-1}$ exists and is uniformly bounded in a neighborhood
of the origin.

Conversely, suppose that
$0\notin\spec \de_1$, and let us show that $0\notin\sigma(A)$.
It follows from the assumption that $(A-z)\Phi(z)=A+z$,
for $z$ in a neighborhood $\cW$ of $0$, where
$\Phi, \Phi^{-1}$ are functions in $H^\infty(\cW,L(H))$.
Moreover, $\Phi(z) h\in \mathcal{D}(A)$ for all $z\in \cW$,
$h\in H$. It follows that $(A-z)^{-1}=\Phi(z)(A+z)^{-1}$ for
$z\in S^\circ_\tht\cap \cW$. Now \eqref{sect_cond} implies
the estimate
$\Vert (z-A)^{-1}\Vert \le C_1\vert z\vert^{-1}$
for all  $z\in\cW$, $z\ne0$.
This first order estimate of the resolvent implies that
$0$ either is in the resolvent set
of $A$ or
is its isolated eigenvalue. The latter contradicts the assumed injectivity  of $A$ (see Introduction) .

This finishes the proof of  the equality
$\sigma(A)=\sigma(\wt A)$ for the case when $\om<\pi/2$.

In the remaining case when $\om\in[\pi/2, \pi)$,
by applying what has been proved already and the results of \cite[Section 3.1]{Haase},
it is easy to prove that
$$
\si(A)=
\big\{
z^2;\; z\in \si(A^{1/2})
\big\}
=
\big\{
z^2;\; z\in \si(\wt A^{\,1/2})
\big\}
=\si(\wt A).
$$
This gives the general case.
\end{proof}


The idea of the following result is that the existence of the $H^\infty$-functional calculus
in two half-planes implies automatically the existence of the $H^\infty$-functional calculus
in their intersection, if one applies the results by Havin,
Nersessian and Ortega-Cerd\'a \cite{Hav-Nersessian, Hav-Ners-Cerda}.

We refer to \cite[Chapter 4, \S 4]{SzNF} for the definition of  purely dissipative operators.

\begin{proposition}
\label{Hav:Ners} Let $A$ be the generator of an analytic semigroup
and assume that $A$ is similar to an $\omega$-accretive operator.
Suppose, moreover, that the operators $\pm e^{\pm i\omega}A$ are
similar to purely dissipative operators. Then $A$ admits an
$H^\infty$-functional calculus in $S_\omega^\circ$.
\end{proposition}

\begin{proof}
Let $\Pi_1$ and $\Pi_2$ be open half-planes such that
$\Pi_1\cap\Pi_2=S^\circ_\omega$ and let $f\in H^\infty(S^\circ_\omega)$.
It follows from \cite[Example 4.1]{Hav-Nersessian} that there is a
constant $C$, depending only on $\omega$ and functions $f_j\in
H^\infty(\Pi_j)$, $j=1,2$, such that
$$
f=f_1+f_2, \qquad \|f_j\|_{H^\infty(\Pi_j)}\le
C\|f\|_{H^\infty(S^\circ_\omega)}.
$$
By the assumption, $A$ has Nagy-Foia\c{s} models in $\Pi_j$. Hence
we can write
\begin{multline*}
\|f(A)\|\le \|f_1(A)\|+\|f_2(A)\|
\\
\le\|f_1\|_{H^\infty(\Pi_1)}+\|f_2\|_{H^\infty(\Pi_2)} \le
2C\,\|f\|_{H^\infty(S^\circ_\omega)},
\end{multline*}
as we wanted to show.
\end{proof}

\noindent  {\bf d. }{\it Duality in the models of Nagy-Foia\c{s} type}.
Note that $A^*$ is also a sectorial operator of type $\om$, and that $A$ admits an $H^\infty$ calculus
iff $A^*$ does, see \cite{AuMcIntNahm}. This is how \textit{four} functional models appear:
the observation and the control models of $A$ and
the observation and the control models of $A^*$.
By \eqref{delta-al}, the characteristic function
$\de_{\al,A^*}$ of $A^*$ is related with
the characteristic function $\de_{\al,A}$ of $A$ via the formula
$$
\de_{\al,A^*}(z)=
\big(\de_{\al,A}(\bar z)\big)^*.
$$

It turns out that there is a certain natural Cauchy duality between
the functional models of $A$ and the functional models of $A^*$.
This point was explained in detail in \cite{Yak}. The Cauchy pairing between the observation model
spaces ${\mathcal H}({\delta}_{\alpha,A})$ and ${\mathcal H}({\delta}_{\alpha,A^*})$
is given by
$$
\langle v, u \rangle_{\delta_\alpha}:={1\over 2\pi
i}\int_{\partial S_\theta} \langle \delta_\alpha(z)v(z),
u(\overline{z})\rangle \,dz, \qquad v\in  {\mathcal
H}({\delta}_\alpha), \,  u\in {\mathcal H}(\tilde{\delta}_\alpha).
$$
The following duality formula holds:
$$
\langle x,y \rangle_{H_A}=
\langle \cO x,\cO_* y \rangle, \quad  x\in H_A,\,y \in H_{A^*},
$$
where $\cO $ is given by \eqref{def_cO} and
$(\cO_* y)(z)=\sqrt{A^*} (z-A^*)^{-1}y$,  $z\in\rho(A^*)$.
See \cite{Yak}, formula (0.1), Proposition 4.2 and \S9.

\

We finish with the following remark.
Let $\tht$ be fixed, and take any $\al, \be \in (0,\frac \pi {2 \tht})$.
Theorem \ref{Obs-mod} implies that
$$
{\mathcal H}(\de_\al)={\mathcal H}(\de_\be).
$$

In general, suppose that
$\De,\De_1\in H^\infty\big(S^\circ_\tht,L(H)\big)$ are
\textit{two-sided admissible functions} (see \cite[\S2]{Yak}). Then, by
\cite[Proposition 11.2]{Yak}, ${\mathcal H}(\De)$=${\mathcal H}(\De_1)$
if and only if there exists an operator-valued function
$\psi$ on $S_\theta^\circ$ with
$\psi,\psi^{-1}\in H^\infty(S_\theta^\circ; L (H))$ such that
$
\De=\psi\De_1.
$
In our situation, the existence of a function $\psi$
such that
$
\delta_\alpha=\psi\delta_\beta
$
can be checked directly.


\def\cprime{$'$}
\def\lfhook#1{\setbox0=\hbox{#1}{\ooalign{\hidewidth
  \lower1.5ex\hbox{'}\hidewidth\crcr\unhbox0}}}

\def\cprime{$'$} \def\lfhook#1{\setbox0=\hbox{#1}{\ooalign{\hidewidth
  \lower1.5ex\hbox{'}\hidewidth\crcr\unhbox0}}}

\vskip1cm



\end{document}